\documentclass[12pt,a4paper,leqno,twoside]{amsart}

\usepackage{latexsym}
\usepackage{amsmath}
\usepackage{amsfonts}
\usepackage{amssymb}
\input xy
\xyoption{all}

\numberwithin{equation}{section}

\theoremstyle{plain}
\newtheorem{thm}{Theorem}[section]
\newtheorem{cor}[thm]{Corollary}
\newtheorem{lemma}[thm]{Lemma}
\newtheorem{prop}[thm]{Proposition}

\theoremstyle{definition}
\newtheorem*{definition}{Definition}
\newtheorem{remark}[thm]{Remark}
\newtheorem{example}[thm]{Example}

\def\i{{\rm i}}

\def\ch{\hbox{\rm ch}}
\def\o{{\mathfrak o}}

\def\su{{\mathfrak su}}
\def\Hom{{\rm Hom}}
\def\End{{\rm End}}
\def\Sp{{\rm Sp}}
\def\SU{{\rm SU}}
\def\PU{{\rm PU}}
\def\SO{{\rm SO}}
\def\Spin{{\rm Spin}}
\def\Spinc{{\rm Spin}^c}
\def\U{{\rm U}}
\def\O{{\rm O}}
\def\rK{\widetilde K}
\def\KO{K{\rm O}}
\def\rKO{\widetilde K{\rm O}}
\def\KSp{K{\rm Sp}}

\def\K{K}
\def\Zz{{\mathbb Z}}
\def\Rr{{\mathbb R}}
\def\Cc{{\mathbb C}}
\def\Qq{{\mathbb Q}}
\def\Hh{{\mathbb H}}
\def\Hl{{\mathbb H}_{\lambda}}
\def\Tt{{\mathbb T}}
\def\Ff{{\mathbb F}}
\def\e{{\rm e}}
\def\st{\mid}
\def\ind{\operatorname{ind}}
\def\bar#1{\overline{#1}}
\def\mod#1{(\hbox{\rm mod}\, {#1})}
\begin{document}

\title{Obstruction theory on $8$-manifolds}

\author[M.~\v Cadek, M.~C.~Crabb and J. Van\v zura]{MARTIN \v CADEK, MICHAEL
CRABB and JI\v R\' I VAN\v ZURA}

%\dedicatory{Version September 15, 2007}

\address{\newline Department of Mathematics, Masaryk University,
Jan\' a\v ckovo n\' am. 2a, 602 00 Brno, Czech Republic}
\email{cadek@math.muni.cz}

\address{\newline Department of Mathematical Sciences, University
of Aberdeen, Aberdeen AB24 3UE, U.K.}
\email{m.crabb@maths.abdn.ac.uk}

\address{\newline Academy of Sciences of the Czech Republic, Institute
of Mathematics, \v Zi\v zkova 22,
616 62 Brno, Czech Republic}
\email{vanzura@ipm.cz}

\date {October 3rd, 2007}

\thanks{Research of the first author supported by the grant
MSM 0021622409 of the Czech Ministry of Education. Research of the third
author supported by the Academy of Sciences
of the Czech Republic, Institutional Research Plan AVOZ10190503, and by the
grant 201/05/2117 of the Grant Agency of the Czech Republic.}

\abstract{This note gives a uniform, self-contained, 
and fairly direct 
approach to a variety of obstruction-theoretic problems on $8$-manifolds.
We give necessary and sufficient cohomological criteria for
the existence of almost complex and almost quaternionic structures
on the tangent bundle and for the reduction of the structure
group to $\U (3)$ by the homomorphism $\U (3) \to \O (8)$ given
by the Lie algebra representation of $\PU (3)$.
}
\endabstract

\maketitle

\section{Introduction}
Let $M$ be a connected, closed, smooth, $8$-dimensional, spin${}^c$ manifold.
We consider, for various compact Lie groups $G$ and
homomorphisms $\rho : G\to \SO (8)$,
the problem of reducing the structure group of the oriented 
tangent bundle, $\tau M$, of $M$ from $\SO (8)$, via $\rho$, to $G$.
In each case we shall obtain, in terms of the cohomology of $M$
and cohomology characteristic classes of $\tau M$,
necessary and sufficient conditions for the reduction.
Amongst the homomorphisms $\rho$ that we examine are:
\begin{enumerate}
\item[(i)]
the standard inclusion $\U (4)\subseteq \SO (8)$;
\item[(ii)]
the homomorphism $\U (3) \to \SO (8)$ determined by the adjoint representation 
of the quotient $\PU (3)$ of $\U (3)$ by its centre;
\item[(iii)]
the standard inclusion $\Sp (2)\subseteq \SO (8)$;
\item[(iv)]
the composition of the double cover and the standard inclusion 
$Sp(2)\times Sp(1)\to Sp(2)\times_{\{\pm 1\}}Sp(1)\subseteq SO(8)$;
\item[(v)]
the composition of the double cover and the standard inclusion
$\Spin (4)\to \SO (4)\subseteq \SO (8)$.
\end{enumerate}
\par\smallskip
The results in cases (i), (iii) and (iv), for the existence
of almost complex and almost quaternionic structures on $M$,
were first obtained, using other methods, 
by Heaps \cite{TH} and by \v Cadek and Van\v zura \cite{CV2}, \cite{CV3},
\cite{CV4}, \cite{CV6}. The result in case (ii) on reduction
to $\U (3)$ through the adjoint representation of $\PU (3)$ was obtained
by Crabb in response to a question of N.J.~Hitchin and F.~Witt; 
see \cite{FW}. 
Conditions for existence of a $4$-field, (v), are also due to
\v Cadek and Van\v zura \cite{CV5}.

More generally, we look at the reduction of the structure group
of any $8$-dimensional oriented real vector bundle $\xi$ over $M$
admitting a spin${}^c$-structure (that is, for which the Stiefel-Whitney
class $w_2(\xi )$ is the reduction of an integral class).
We fix a spin${}^c$ structure on $M$,
with characteristic class $c\in H^2(M;\, \Zz )$ (lifting $w_2(M)$),
and an orientation for the vector bundle $\xi$.

It is, at first sight, somewhat surprising that one should be able to
obtain purely cohomological criteria for reduction of the structure group.
However, there are certain special features of the problem that make
its solution tractable (and, even, rather easy).

\smallskip
\par\noindent {\bf 1.1}.
Two oriented $8$-dimensional real vector bundles $\xi$ and $\xi'$ over $M$
such that $[\xi ]=[\xi']\in\KO^0(M)$
are isomorphic as oriented bundles if and only if $e(\xi )[M]=e(\xi')[M]
\in\Zz$.

This reduces the obstruction theory to a $\KO$-theoretic, stable problem.

\smallskip
\par\noindent {\bf 1.2}.
The spin${}^c$ structure for $M$
allows us to split off the `top cell' in real and complex
$K$-theory. Let $B\subseteq M$ be an embedded open disc of dimension $8$,
that is, a tubular neighbourhood of a point. 
Then we have short exact sequences:
$$
\begin{matrix}
0&\to&\Zz = \KO^0(M,M-B)&\xrightarrow[\leftarrow - -]{}
&\KO ^0(M) &\to &\KO^0 (M-B)&\to &0\\
&&\operatorname{id}\,\,\downarrow\,\,\cong\qquad\,\,\,  & & \downarrow && \downarrow && \\
0&\to &\Zz = \K^0(M,M-B)&\xrightarrow[\leftarrow - -]{}
&\K ^0(M) &\to &\K^0 (M-B)&\to &0
\end{matrix}
$$
split by the spin${}^c$ index:
$$
\ind : \K^0(M) \to \Zz .
$$
Note that $\KO^1(M,M-B)$ and $\K^1(M,M-B)$ are zero. 
Thus the splitting of $K^0(M)$ induces a splitting
of $\KO^0(M)$ as a direct sum $\KO^0(M-B)\oplus\Zz$, and a real
vector bundle $\xi$ over $M$ is determined stably by its restriction to
$M-B$ and the image of $[\xi ]\in\KO^0(M)$ under the spin${}^c$ index.

This leaves us with an obstruction theory problem on the essentially
$7$-dimensional space $M-B$.

\smallskip
\par\noindent {\bf 1.3}.
In a cell decomposition of $M$ as a finite complex
(in which $B$ is an open cell) the restriction map
$$
\KO^0(M-B) \to \KO^0(M^{(4)})
$$
to the $4$-skeleton is injective.
For an argument using the Atiyah-Hirzebruch spectral sequence shows
that the restriction map $\KO^0(M-B)\to\KO^0(M^{(7)})$
is an isomorphism, since $H^8(M-B;\,\Zz )=0$ and
$\KO^1(M-B,M^{(7)})=0$,
and the restriction $\KO^0(M^{(7)})\to\KO^0(M^{(4)})$
is injective, because 
$\KO^0(M^{(5)},M^{(4)})=0$, $\KO^0(M^{(6)},M^{(5)})=0$ and
$\KO^0(M^{(7)},M^{(6)})=0$. 

This means that we can distinguish stable bundles over $M-B$ by 
calculations in dimension $4$ and below.

\smallskip
Conditions for the existence of a complex structure on the vector bundle
$\xi$ are derived in Section 4 and stated as Proposition \ref{4A}.
The result on the reduction of the structure group from $\SO (8)$ to
$\U (3)$ is given as Proposition \ref{6A}.
Sections 7, 8 and 10 deal with almost quaternionic structures; 
the main results appear as
Propositions \ref{8B}, \ref{9A} and \ref{13C}.
Reduction to $\Spin (3)$ and to $\Spin (4)$ is considered in Section 9.

Almost fifty years ago
Hirzebruch and Hopf \cite{FH-HH} investigated the reduction
of the structure group of an oriented $4$-manifold
from $\SO (4)$ to $\SO (2)\times \SO (2)$
and to $\U (2)$.
This paper follows in the tradition of their work.

\section{The spin characteristic class in dimension $4$}
In this section we recall some fairly standard facts about 
the spin characteristic class in dimension $4$ and the classification
of spin and spin${}^c$ bundles in low dimensions.
We shall write $\rho_2 : H^*(M;\, \Zz ) \to H^*(M;\,\Ff_2)$ for reduction
(mod $2$). 

An elementary computation shows that
the Chern classes $c_2$ and $c_3$ of a complex vector bundle
with $c_1=0$ satisfy:
\begin{equation}\label{2A}
{\rm Sq}^2 (\rho_2 c_2) =\rho_2 c_3.
\end{equation}
Let $F$ be the homotopy fibre of
$$
{\rm Sq}^2 \circ\rho_2 + \rho_2: K(\Zz ,4)\times K(\Zz ,6) \to
K(\Ff_2,6) .
$$
By (\ref{2A}), we can lift the map 
$(c_2,c_3): B\SU (\infty ) \to K(\Zz ,4)\times K(\Zz ,6)$
to a map $B\SU (\infty )\to F$.
\begin{lemma}\label{2B}
The map $B\SU (\infty ) \to F$ above induces an isomorphism
in homotopy groups $\pi_i (B\SU (\infty )) \to \pi_i(F)$ for
$i\leq 7$ (and a surjection for $i=8$).
\end{lemma}
\begin{proof}
We have $\pi_4(F)=\Zz$, $\pi_6(F)=2\Zz$, and $\pi_i(F)=0$ otherwise.
The Chern class $c_2 : \rK^0(S^4)=\Zz \to H^4(S^4;\, \Zz )=\Zz$ is an
isomorphism, and $c_3 :\rK^0(S^6)=\Zz\to H^6(S^6;\,\Zz )=\Zz$ is 
multiplication by $2$.
\end{proof}

\smallskip
Recall that $M$ is a closed, connected, $8$-dimensional, spin${}^c$-manifold
and that $B\subseteq M$ is an open disc.
The complex $K$-group $\K^0(M-B)=[(M-B)_+;\, \Zz \times B\U (\infty )]$
splits as a direct sum
$$
K^0(M-B)  =\Zz \oplus H^2(M;\, \Zz )\oplus [(M-B)_+;\, B\SU (\infty )],
$$
where the summand $\Zz =[(M-B)_+;\, \Zz ]$, 
corresponding to the trivial bundles, is given by the dimension and the summand
$H^2(M;\, \Zz )=[(M-B)_+;\, B\U (1)]$, 
corresponding to the line bundles, is given by $c_1$.
(The subscript $+$ denotes a disjoint basepoint, and the brackets $[-;-]$
indicate the set of pointed homotopy classes.)

According to Lemma \ref{2B}, 
the map $[(M-B)_+;\, B\SU (\infty )]\to [(M-B)_+;\, F]$
is surjective (and, in fact, bijective). 
This completes the description of the Chern classes of
$\SU$-bundles over $M-B$. (See \cite{LMW} for a similar description
in the case of real vector bundles.)
\begin{prop}
The image of the mapping
$$
(c_2,c_3) : [(M-B)_+;\, B\SU (\infty ) ] 
\to H^4(M;\,\Zz )\oplus H^6(M;\, \Zz )
$$
is the set $\{ (u,v) \st {\rm Sq}^2\rho_2(u) =\rho_2(v)\}$.
%\qed
\end{prop}

By adding (or subtracting) a complex line bundle with $c_1=l$ one
obtains the following generalization.
\begin{cor}\label{2C}
The image of the mapping
$$
(c_1,c_2,c_3) : [(M-B)_+;\, B\U (\infty ) ] 
\to H^2(M;\,\Zz )\oplus H^4(M;\,\Zz )\oplus H^6(M;\, \Zz )
$$
is the set $\{ (l,u,v) \st {\rm Sq}^2\rho_2(u)+\rho_2(lu) =\rho_2(v)\}$.
%\qed
\end{cor}

\begin{lemma}
The inclusion map
$$
\SU (\infty ) \to \Spin (\infty ).
$$
induces an isomorphism $\pi_i(B\SU (\infty )) \to \pi_i(B\Spin (\infty ))$
for $i\leq 5$.
\end{lemma}
\begin{proof}
 For restriction from $\Cc$ to $\Rr$ gives an isomorphism
$\K^{-4}(*)=\Zz \to \KO^{-4}(*)=\Zz$;
the homotopy groups are zero for $i<4$ and for $i=5$.
\end{proof}
\begin{definition}
The spin characteristic class in $H^4(B\Spin (\infty );\,\Zz )$
corresponding to $-c_2\in H^4(B\SU (\infty );\, \Zz )$
is denoted by $q_1$.
For an oriented real vector bundle $\xi$ admitting a spin structure,
that is,  such that $w_2(\xi )=0$, the characteristic class $q_1$ is,
as we show below, 
independent of the choice of spin structure and we 
write it as $q_1(\xi )\in H^4(M;\,\Zz )$.
\end{definition}

To see the independence we may argue as follows.
Over the $5$-skeleton there is a complex vector bundle $\zeta$ with
$c_1(\zeta )=0$ that is isomorphic to $\xi\, |\, M^{(5)}$ as a real bundle. 
We must verify that $c_2(\zeta )$ does not depend on the choice of 
an $\SU$-structure.
If $\zeta'$ is another such bundle, then $\zeta -\zeta'$ is stably
trivial as a real bundle and so,
considered as a map $M^{(5)}\to B\SU$, lifts to $\SO /\SU$.
We have to show that $c_2(\zeta -\zeta')=0$. But $\rho_2 :H^4(\SO /\SU ;\, \Zz )
\to H^4(\SO /\SU;\, \Ff_2)=\Ff_2$ is an isomorphism and
$\rho_2(c_2(\zeta -\zeta '))= w_4(\zeta -\zeta')=0$.. 

From the definition it is immediate that, for
a spin bundle $\xi$, we have $2q_1(\xi ) = p_1(\xi )$ and 
$\rho_2(q_1(\xi ))=w_4(\xi )$.
Note also that $q_1$ is additive:
$q_1(\xi\oplus\xi')=q_1(\xi )+q_1(\xi ')$ for two bundles $\xi$ and $\xi'$
admitting spin structures.
A complex vector bundle $\zeta$ admits a spin structure if and only if
$c_1(\zeta )$ is divisible by $2$, say equal to $2m$ and then
$q_1(\zeta )=2m^2-c_2(\zeta )$ (which does not depend on the choice of $m$).

We next extend this definition formally to a characteristic class for
spin${}^c$-bundles.
Let $\xi$ now be a real vector bundle admitting a spin${}^c$ structure.
Thus $w_2(\xi )$ is the reduction (mod $2$) $\rho_2(l)$ of an integral
class $l\in H^2(M;\,\Zz )$.
\begin{definition}
For an orientable vector bundle $\xi$ 
and class $l\in H^2(M;\,\Zz )$ such that
$\rho_2(l)=w_2(\xi )$ we define 
$$
q_1(\xi ;\, l) =q_1(\xi-\lambda )\in H^4(M;\, \Zz ),
$$
where $\lambda$ is a complex line bundle with $c_1(\lambda )=l$.
\end{definition}

It is elementary to verify that:
\begin{align*}
2q_1(\xi ;\, l)&=p_1(\xi )  - l^2;\qquad
\rho_2(q_1(\xi ;\, l))=w_4(\xi );\\
q_1(\xi ;\, l+2m) &=q_1(\xi ;\, l)-2lm-2m^2
\hbox{\rm \  for $m\in H^2(M;\, \Zz )$,}
\end{align*}
and, if $\zeta$ is a complex bundle with $c_1(\zeta )=l$, then
$q_1(\zeta ;\, l) = - c_2(\zeta )$.

This invariant is enough to distinguish oriented vector bundles over $M-B$.
\begin{prop}\label{2D}
Consider two oriented (stable) vector bundles $\xi$ and $\xi'$
over $M-B$ with $w_2(\xi )=w_2(\xi')=\rho_2(l)$, where
$l\in H^2(M;\, \Zz )$.
Then $\xi$ and $\xi'$ are stably isomorphic if and only if
$q_1(\xi ;\, l)=q_1(\xi';\, l)$.
\end{prop}
\begin{proof}
For $\xi -\xi'$ is a spin bundle with $q_1(\xi -\xi')=0$. So
$\xi -\xi'$ is trivial over $M^{(4)}$, and hence, by (1.3),
over $M-B$. 
\end{proof}
\begin{remark}
With the benefit of hindsight it is now easy to rederive the main results 
of Hirzebruch and Hopf \cite{FH-HH}.
Let $\eta$ be an oriented $4$-dimensional real vector bundle with
$w_2(\eta )$ the reduction of an integral class
over a closed, connected, oriented $4$-manifold $N$.
Then $\eta$ is isomorphic to a direct sum $\mu_1\oplus\mu_2$ of
oriented $2$-dimensional bundles with $e(\mu_i)=m_i$
if and only if
$w_2(\eta )=\rho_2(l)$, where $l=m_1+m_2$,
$e(\eta )=m_1m_2$ and
$p_1(\eta ) = m_1^2 +m_2^2$ ($=l^2-2m_1m_2$).
For we have only to compare the classes $q_1(\eta ;\, l)$ with
$q_1(\mu_1\oplus\mu_2;\, l)$ and $e(\eta )$ with $e(\mu_1\oplus \mu_2)$.
Similarly, $\eta$ has a complex structure (compatible with its given
orientation) with $c_1=l$ if and only if $w_2(\eta )=\rho_2(l)$
and $p_1(\eta )+2e(\eta )=l^2$.
\end{remark}

\section{The spin${}^c$ index}
We turn next to the cohomological description of the fundamental splitting
(1.2).
The spin${}^c$ index is given by
\begin{equation}\label{3A}
y\in \K^0(M)\,\mapsto\, 
\ind\, y = (\e^{c/2}\hat{\mathcal A}(\tau M)\hbox{\rm ch}(y))[M]\in\Zz .
\end{equation}
Here the Chern character, ${\rm ch}(y)$, of $y$ has the explicit
expansion:
\begin{equation}\label{3B}
\begin{aligned}
\hbox{\rm ch}(y)&=\dim\, y+ c_1 + (c_1^2-2c_2)/2 + (c_1^3-3c_1c_2+3c_3)/6+
\\&\qquad (c_1^4-4c_1^2c_2+4c_1c_3+2c_2^2-4c_4)/24 + \ldots ,
\end{aligned}
\end{equation}
where the $c_i$ are the Chern classes of $y$, and
\begin{equation*}
\hat{\mathcal A}(\tau M) = 1 - p_1(\tau M)/24 + 
(-4p_2(\tau M) +7p_1^2(\tau M))/5760 +\cdots .
\end{equation*}
The Chern character of the complexification of a real class $x\in \KO ^0(M)$ 
is written in terms of the Pontrjagin classes $p_i$ of $x$ as:
\begin{equation}\label{3D}
\dim\, x+ p_1 + (p_1^2-2p_2)/12 + \ldots ;
\end{equation}
for $p_1(x)=-c_2(y)$ and $p_2(x)=c_4(y)$, where $y$ is the complexification
of $x$.

\begin{remark}\label{3C}
From the Universal Coefficient theorem and Poincar\'e duality for
the oriented $8$-dimensional manifold $M$ one obtains a short exact sequence
$$
0 \to T(H^2(M;\,\Zz ))\otimes\Ff_2 \to H^2(M;\,\Ff_2)
\to \Hom (H^6(M;\, \Zz ),\Ff_2) \to 0, 
$$
in which $T$ denotes the torsion subgroup of an abelian group.

When  $w_2(M)\cdot H^6(M;\,\Zz )=0$,
or equivalently, by the Wu formula, when
${\rm Sq}^2\rho_2 H^6(M;\, \Zz )$ $=0$, we may choose 
a spin${}^c$-structure for $M$ for which the characteristic class
$c$ is torsion and so disappears from the rational cohomology formulae. 

On the other hand, if $w_2(M)\cdot H^6(M;\,\Zz )\not=0$, then there
is, for any chosen class $c$, an element $t\in H^6(M;\, \Zz )$
such that $(c\cdot t)[M]\in\Zz$ is odd.
\end{remark}

We can use the splitting to extend the description (\ref{2C}) of the Chern
classes of complex vector bundles over $M-B$ to bundles over $M$.
\begin{prop}\label{3E}
Consider cohomology classes
$l\in H^2(M;\, \Zz )$, 
$u\in H^4(M;\, \Zz )$, $v\in H^6(M;\, \Zz )$, $w\in H^8(M;\, \Zz )$.
According to {\rm (\ref{2C})},
there is a $4$-dimensional complex vector bundle $\zeta$ over
$M-B$ with $c_1(\zeta )=l$, $c_2(\zeta )=u$ and
$c_3(\zeta )=v$
if and only if ${\rm Sq}^2 \rho_2u =\rho_2 (v+lu)$.
In that case 
$$
\frac{1}{2}(u(u+q_1(\tau M;\, c) -2l^2 -3cl -c^2) +(2l +3c)v)[M] \in\Qq
$$
is an integer, $I$ say,
and $I\, (\text{\rm mod}\, 6)$ is independent of the choice of $c$.
Such a bundle extends over $M$ to a complex vector bundle $\zeta$ with
$c_4(\zeta )=w$
if and only if 
$$
w[M] \equiv I\, (\hbox{\rm mod}\, 6).
$$
\end{prop}
%Using $p_1(\tau M)$ instead of $q_1(\tau M;\, c)$ we get
%$$
%I=\frac{1}{4}(u(2u+p_1(\tau M)-4l^2-6cl-3c^2)+(4l+6c)v)[M].
%$$ 
\begin{proof}
We observe, first of all, that $4$-dimensional complex vector bundles
over an $8$-dimensional space lie in the stable range, and so the
problem is $K$-theoretic.
By (1.2), the stable complex bundles $y\in \K^0(M)$
over $M$ extending a given bundle over $M-B$ are classified by the integer 
$\ind y\in\Zz$.
A computation using (\ref{3A}) and (\ref{3B}) shows that, for a bundle
$\zeta$ with the given Chern classes and a complex line bundle
$\lambda$ with $c_1(\lambda )=l$, we have
$$
\ind (\zeta -\Cc^4) -\ind (\lambda -\Cc ) = \frac{1}{6}(I-w[M]).
$$
The result follows.
\end{proof}

In the same way, we may use the splitting of $\KO^0(M)$ to distinguish
oriented real vector bundles over $M$.
\begin{prop}\label{3F}
Let $\xi$ and $\xi'$ be oriented $8$-dimensional real vector bundles
over $M$ with $w_2(\xi )=w_2(\xi')=\rho_2(l)$, where $l\in H^2(M;\, \Zz )$.
Then $\xi$ and $\xi'$ are isomorphic if and only if
$$
\text{\rm a)\quad}
q_1(\xi ;\, l)=q_1(\xi';\, l),\qquad
\text{\rm b)\quad}
p_2(\xi )=p_2(\xi'),\qquad
\text{\rm c)\quad}
e(\xi )=e(\xi').
$$
\end{prop}
\begin{proof}
Condition a) is necessary and sufficient for the restrictions of
$\xi$ and $\xi'$ to $M-B$ to be isomorphic.  It implies, in particular,
that $p_1(\xi )=p_1(\xi')$.
And then $\ind (\Cc\otimes\xi) =\ind (\Cc\otimes\xi')$ if and only
if $p_2(\xi )=p_2(\xi')$, from (\ref{3A}) and (\ref{3D}). 
Hence a) and b) are necessary and sufficient conditions for
$\xi$ and $\xi'$ to be stably isomorphic, by (1.2).
Finally, equality of the Euler classes, (1.1), gives isomorphism
as oriented vector bundles.
\end{proof}

\section{Reduction to $\U (4)$ : a complex structure}
Consider an oriented real vector bundle $\xi$ of dimension $8$ over
$M$ such that $w_2(\xi )$ lifts to an integral class $l\in H^2(M;\, \Zz )$.
We shall determine necessary and sufficient conditions for $\xi$ to
admit a complex structure with first Chern class $c_1=l$.
Notice that $p_2(\xi )-q_1^2(\xi ;\, l)$ is divisible by $2$
in $H^8(M;\,\Zz )$,
since $\rho_2(p_2(\xi )-q_1^2(\xi ;\, l))=w_4^2(\xi )-w_4^2(\xi )=0$.
\begin{prop}\label{4A}
{\rm (Heaps, \v Cadek and Van\v zura)}
Suppose that $\xi$ is an oriented $8$-dimensional real vector bundle
over $M$ admitting a spin${}^c$ structure.
Let $l\in H^2(M;\, \Zz )$ be a class such that $w_2(\xi )=\rho_2(l)$.
Then the structure group $\SO (8)$ of $\xi$ admits a reduction to
$\U (4)$ with $c_1=l$ if and only if there is a class $v\in H^6(M;\, \Zz )$
with $\rho_2(v)=w_6(\xi )$ such that
$$
\text{\rm a)\quad}
\frac{1}{2} (p_2(\xi )- q_1^2(\xi ;\, l))[M]  \equiv J\, (\text{\rm mod}\, 2)
\text{\quad and \quad}
\text{\rm b)\quad}
p_2(\xi )-q_1^2(\xi ;\, l)=2(e(\xi )-lv),
$$
where
$$
J=\frac{1}{2}(q_1(\xi ;\, l)(q_1(\xi ;\, l)-q_1(\tau M;\, c)+2l^2+3lc+c^2)
+3cv)[M]
$$
(which is an integer if $\rho_2(v)=w_6(\xi )$).
\end{prop}
%Using $p_1(\tau M)$ instead of $q_1(\tau M;\, c)$ we have
%$$
%J=
%\frac{1}{4}((q_1(\xi ;\, l)(2q_1(\xi ;\, l)-p_1(\tau M)+4l^2+6lc+3c^2) 
%+6cv)[M].
%$$
\begin{remark}
The condition a), given that $\rho_2(v)=w_6(\xi )$, is necessary and
sufficient for the existence of a stable complex structure on $\xi$,
that is, a complex structure on $\xi\oplus\Rr^2$.
(More precisely, if $\xi$ has a stable complex structure outside a
finite subset of $M$, then the sum of the local obstructions
in $\pi_7(\SO (\infty )/\U (\infty ))=\Zz /2$ to extending the stable
structure to $M$ is equal to $\frac{1}{2}(p_2(\xi )-q_1^2(\xi ;\, l))[M]-J$
(mod $2$).)
If the class $c$ can be chosen to be torsion, then the expression for
$J$ simplifies, and, in particular, $J$ does not depend on $v$. 
If not, then, as was observed in (\ref{3C}), there is a
class $t$ such that $(ct)[M]$ is odd, and $v$ can be modified by
addition of a multiple of $2t$ to satisfy a). Thus, if
$w_2(M)H^6(M;\,\Zz )\not=0$, the only condition required
for the existence of a stable complex structure with $c_1=l$ is that
$w_6(\xi )$ should be the reduction of an integral class.
(Compare \cite{TH}.)
\end{remark}
\begin{proof}
Suppose that there is a $4$-dimensional complex vector bundle $\zeta$,
with Chern classes $c_1=l$, $c_2=u$, $c_3=v$, $c_4=w$, that
is stably isomorphic to $\xi$.
Then 
$$
\rho_2(v)=w_6(\xi ), \quad u=-q_1(\xi ;\, l)\quad \text{\rm and}\quad
2w=p_2(\xi )-q_1^2(\xi ;\, l) + 2lv,
$$
because $p_2(\xi )=c_4(\zeta\oplus\bar\zeta )=2c_4(\zeta )-2c_1(\zeta )
c_3(\zeta )+ c_2^2(\zeta )$.

Necessary and sufficient conditions for the existence of a stable
complex structure on $\xi$ are, therefore, given by (\ref{3E})
and (\ref{3F}) with $J= I - (lv)[M]$.
(Notice that ${\rm Sq}^2 w_4(\xi )=w_6(\xi )+w_2(\xi )w_4(\xi )$.)
This reduces to the condition a), except that congruence is required
(mod $6$).
But the congruence is automatically satisfied (mod $3$). 
For the integrality of $\ind (\Cc\otimes_\Rr\xi -\Cc^8)-
\ind (\Cc\otimes_\Rr\lambda -\Cc^2)$, where $\lambda$ as before is a complex
line bundle with $c_1=l$, gives 
\begin{equation}\label{4B}
(p_2(\xi )-q_1^2(\xi ;\, l))[M] \equiv 
(q_1(\xi ;\, l)(q_1(\xi ;\, l)-q_1(\tau M;\, c) +2l^2 +c^2))[M]
\,\, (\text{\rm mod}\, 6).
\end{equation}
(It is, of course, clear that there can be no $3$-primary obstruction to
the existence of a stable complex structure on $\xi$, because
$2\xi=\xi\oplus\xi$ admits a complex structure as $\Cc\otimes_\Rr \xi$.)

Condition b) expresses the equality of the Euler classes required for
isomorphism rather than stable isomorphism.
\end{proof}

\section{Reduction to $\Spinc (6)$ : a $2$-field}
In this section we consider reduction to 
$$
\Spinc (6) \to \SO (6) \to \SO (8),
$$
where the first homomorphism is the canonical projection 
and the second is the standard inclusion.

There is an isomorphism
$$
\{ (z,g)\in \Tt\times \U (4) \st z^2=\det g\} \to \Spinc (6).
$$

Suppose that $\zeta$ is a $4$-dimensional $\Cc$-vector bundle over
$M$, and $\lambda$ a complex line bundle, and that we are given an
isomorphism $\Lambda^4\zeta \to \lambda\otimes\lambda$. Then
$\lambda^*\otimes\Lambda^2\zeta$ has a real structure as $\Cc\otimes_\Rr \eta$,
where $\eta$ is a $6$-dimensional real vector bundle over $M$.
The trivialization of $\Lambda^6 (\lambda^*\otimes \Lambda^2\zeta )=
\lambda^{-6}\otimes (\Lambda^4\zeta )^3$ gives an orientation of $\eta$.
\begin{lemma}\label{5A}
The characteristic classes of $\eta$ are as follows:
$$
\begin{matrix}
w_2(\eta )=\rho_2(l),\quad&
q_1(\eta ;\, l)= l^2-c_2(\zeta ), \\
e(\eta )=c_3(\zeta )+lq_1(\eta ;\, l),\quad &
p_2(\eta )= q_1^2(\eta ;\, l ) +2l e(\eta) -4c_4(\zeta ),
\end{matrix}
$$
where $l=c_1(\lambda )$ and the orientation of $\eta$ is chosen appropriately.
\end{lemma}
\begin{proof}
Over the $7$-skeleton,
the $4$-dimensional complex vector bundle $(\lambda^*)^2\otimes\zeta$
has a non-zero section, and hence 
we can write $\zeta =\zeta'\oplus \lambda^2$, with $\Lambda^3\zeta'=\Cc$. 
So 
$\lambda^*\otimes\Lambda^2\zeta = 
(\lambda^*\otimes\Lambda^2\zeta')\oplus (\lambda\otimes\zeta')$
and $\eta$ is the underlying real bundle of $\lambda\otimes\zeta'$.
This allows the computation of $w_2$, $q_1$ and $e$. 
The second Pontrjagin class is $c_4(\lambda^*\otimes\Lambda^2\zeta )$.
\end{proof}
\begin{prop}\label{5B}
Let $\xi$ be an oriented $8$-dimensional real vector bundle
over $M$ with $w_2(\xi )=\rho_2(l)$, $l\in H^2(M;\, \Zz )$. 
Let $v\in H^6(M;\,\Zz )$.
Then $\xi$ admits a  reduction of structure group to $\Spinc (6)$
with spin${}^c$ characteristic class $l$ and Euler class $v$
if and only if the following conditions are satisfied.
\begin{align*}
e(\xi )&=0;\quad \rho_2(v)=w_6(\xi );\\
(\frac{1}{2}(p_2(\xi )-q_1^2(\xi ;\, l)) &+
q_1(\xi ;\, l)(q_1(\xi ;\, l) -q_1(\tau M;\, c) +2l^2+3lc +c^2)
+ 3(l+c)v)[M]\\
 & \equiv 0\,\, (\text{\rm mod}\, 4).
\end{align*}
\end{prop}
%Using $p_1(\tau M)$ instead of $q_1(\tau M;\, c)$ the last expression is
%$$
%(\frac{1}{2}(p_2(\xi )-q_1^2(\xi ;\, l)) +\frac{1}{2}
%(q_1(\xi ;\, l)(2q_1(\xi ;\, l) -p_1(\tau M) +4l^2+6lc +3c^2)
%+ 6(l+c)v)[M].
%$$
\begin{proof}
We require a $4$-dimensional complex vector bundle $\zeta$ over $M$
with first Chern class $c_1(\zeta )=2l$ 
such that the associated bundle $\eta$ satisfies:
$$
w_2(\xi )=w_2(\eta ); \quad
q_1(\xi;\, l)=q_1(\eta ;\, l);\quad
v= e(\eta );\quad
p_2(\xi )=p_2(\eta ).
$$
This will guarantee that $\xi$ and $\eta\oplus\Rr^2$ are stably isomorphic.
The vanishing of the Euler class $e(\xi )$ will then be enough for
the two vector bundles to be isomorphic.

The conditions for the existence of $\zeta$ are given in  Proposition
\ref{3E}.
One has first
$$
\hbox{\rm Sq}^2(\rho_2(l^2-q_1(\xi ;\, l))) = \rho_2(v-lq_1(\xi ;\, l)),
$$ 
that is, ${\rm Sq}^2(w_2^2(\xi )+ w_4(\xi ))=\rho_2(v)+w_2(\xi)w_4(\xi)$,
which reduces to the obvious condition that $\rho_2(v)=w_6(\xi )$.
In the notation of Proposition \ref{3E} we find that
\begin{multline*}
2I-2w[M]=(q_1(\xi ;\, l)(q_1(\xi ;\, l) -q_1(\tau M;\, c) +2l^2 +3lc+c^2)
+ l^2q_1(\tau M;\, c)\\
- l^2(7l^2+6cl+c^2) +
(4l+3c)v)[M] 
+(\frac{1}{2}(p_2(\xi )-q_1^2(\xi ;\, l))-lv)[M] \in 12\Zz.
\end{multline*}
By considering the special case in which $\xi = \lambda\oplus\Rr^6$,
with $q_1(\xi ;\, l)=0$ and $p_2(\xi )=0$, when we may take $v=0$,
we see that
$$
(l^2(q_1(\tau M;\, c) - (7l^2+6cl+c^2)))[M]\in 12\Zz .
$$ 
%This follows from the integrality of
%$\ind (\Cc\otimes\lambda -\Cc^2)$, which gives
%$(l^4 +c^2l^2-l^2q_1(\tau M;\, c))[M]\in 12\Zz$
%and $(l^4+cl^3)[M]\in 2\Zz$,
%which follows from 
%${\rm Sq}^2(w_2^3(\xi ))=w_2^4(\xi )$.
This leads to the stated congruence, but (mod $12$) instead of (mod $4$).
However the congruence automatically holds (mod $3$), by (\ref{4B}).
\end{proof}
\begin{remark}
Given a complex line bundle $\mu$ over $M$,
there is a similar necessary and sufficient condition for $\xi$ to
split as a sum $\eta\oplus\mu$.
There are other methods when $w_2(\xi )=w_2(M)$ or
$w_2(\mu )=w_2(\xi )+w_2(M)$; see \cite{MCC-BS} and 
\cite{CV1}.
The problems of reduction to $\U (4)$ and $\Spinc (6)$ are related,
or rather the $\SU (4)$ and $\Spin (6)$ reductions - the cases
in which $l=0$. The correspondence is explained in \cite{CV2}.
\end{remark}

\section{Reduction to $\U (3)$ : the adjoint representation}
For a finite dimensional complex Hilbert space $V$, let us write
$\su (V)$ for the Lie algebra of the special unitary group. 
Thus $\dim_\Rr \su (V) =(\dim_\Cc V)^2 -1$.
In particular, taking $V=\Cc^3$, we have an $8$-dimensional
real representation $\su (\Cc^3)$ of $\U (3)$, and so, up to
conjugation, a homomorphism $\rho : \U (3) \to \SO (8)$.
(In fact, since $\rho$ factors through the projective unitary
group $\PU (3)$, which is a quotient of $\SU (3)$ by the central
subgroup of order $3$, the homomorphism must lift to $\U (3)\to\Spin (8)$.)
In this section we investigate the reduction of the structure group
of $\xi$ over $M$ to $\U (3)$.

So suppose that $\zeta$ is a $3$-dimensional complex (unitary) vector
bundle over $M$. We form the associated Lie algebra bundle
$\su (\zeta )$ of real dimension $8$.
\begin{lemma}
The vector bundle $\su (\zeta )$ is spin and its characteristic
classes are as follows:
$$
\begin{matrix}
w_2(\su (\zeta )) =0, &   e(\su (\zeta )) =0, \\
q_1(\su (\zeta )) = c_1^2(\zeta )-3c_2(\zeta ), &
p_2(\su (\zeta )) = c_1^4(\zeta )-6c_1^2(\zeta )c_2(\zeta )
+9c_2^2(\zeta ), \\
w_4(\su (\zeta ))= \rho_2(c_1^2(\zeta )+c_2(\zeta )),&
w_6(\su (\zeta ))= \rho_2(c_1(\zeta )c_2(\zeta )+c_3(\zeta)).
\end{matrix}
$$
\end{lemma}
\begin{proof}
We may calculate in the torsion-free cohomology of $B\U (3)$ and
so reduce to the case that $\zeta = \mu_1\oplus\mu_2\oplus
\mu_3$ is a sum of $3$ complex line bundles. In that case we
have the explicit description (using a bar for the complex conjugate):
$$
\su (\zeta ) = \i\Rr^2 \oplus \left((\bar{\mu_1}\otimes\mu_2)
\oplus (\bar{\mu_1}\otimes\mu_3)
\oplus (\bar{\mu_2}\otimes\mu_3)\right)
$$
(as a real bundle).
The computations are then straightforward. 
\end{proof}

The standard method leads to necessary and sufficient conditions for
reduction of the structure group.
\begin{prop}\label{6A}
Let $\xi$ be an oriented $8$-dimensional real vector bundle over $M$ with
$w_2(\xi )=0$. 
Consider cohomology classes
$l\in H^2(M;\, \Zz )$, $u\in H^4(M;\, \Zz )$, $v\in H^6(M;\, \Zz )$.
Then there is a $3$-dimensional complex
vector bundle $\zeta$ over $M$ such that $\su (\zeta )\cong \xi$, with
$c_1(\zeta )=l$, $c_2(\zeta )=u$ and $c_3(\zeta )=v$, if and only if
$$
q_1(\xi )=-3u+l^2;\quad 
\rho_2(v-lu)=w_6(\xi );\quad
p_2(\xi )= q_1^2(\xi );\quad e(\xi )=0;
$$
and
$$
\frac{1}{2}(u(u+q_1(\tau M ;\, c)-c^2) +(2l+3c)(v-lu))[M] \equiv 0\,\,
({\rm mod}\,  6).
$$
\end{prop}
(It is to be understood in the statement that the expression
in the last line is integral if $\rho_2(v-lu)=w_6(\xi )$.)
%Using $p_1(\tau M)$ instead of $q_1(\tau M;\, c)$ the last expression
%is
%$\frac{1}{4}(u(2u+p_1(\tau M)-3c^2) +(4l+6c)(v-lu))[M]$.
\begin{proof}
We require first the existence of a complex vector bundle $\zeta$ as in
Proposition \ref{3E} with $w=0$, so that $\zeta$ admits a nowhere zero 
cross-section.
Then the conditions that $q_1(\xi )=l^2-3u$ and 
$p_2(\xi ) = l^4-6l^2u+9u^2$ ensure that $\su (\zeta )$ and $\xi$ are
stably isomorphic. 
Recall that $\rho_2(q_1(\xi ))=w_4(\xi )$, so that
${\rm Sq}^2\rho_2(u) = {\rm Sq}^2 w_4(\xi ) =w_6(\xi )$,
because $w_2(\xi )=0$.

Finally,
equality of the Euler classes $e(\su (\zeta ))=0=e(\xi )$
is necessary and sufficient for the two bundles to be isomorphic.
\end{proof}
\begin{remark}
If $\mu$ is a complex line bundle, then
$\su (\mu\otimes\zeta )=\su (\zeta )$.
Hence, 
the conditions above must be invariant under the transformation
$\zeta\mapsto \mu\otimes\zeta$, that is:
$l\mapsto l+3m$, $u\mapsto u+2lm+3m^2$, $v\mapsto v+um+lm^2+m^3$,
where $c_1(\mu )=m$.
\end{remark}

The conditions simplify if the manifold $M$ is spin ($c=0$) and
we look for an $\SU (3)$-structure ($l=0$) on $\xi$.
\begin{cor}
Suppose that $M$ is a spin manifold. Then the bundle $\xi$
admits an $\SU (3)$-structure if and only if
$w_6(\xi )$ is the reduction of an integral class,
$e(\xi )=0$, and there is a class $u\in H^4(M;\, \Zz )$ such that
$$
\text{$q_1(\xi )=-3u$,\quad
$p_2(\xi )= 9u^2$ and
$\frac{1}{2}(u(u+q_1(\tau M)))[M]\equiv 0\,\, ({\rm mod}\, 6)$.}
%\Qed
$$
\end{cor}
%Equivalently,
%$\frac{1}{4}(u(2u+p_1(\tau M)))[M]\equiv 0\,\, ({\rm mod}\, 6)$.

For the special case in which $\xi =\tau M$, see \cite{FW}
(although the argument there purporting to show that $w_6(M)=0$ 
seems to be incorrect).

\section{Quaternionic bundles}
Modules over a quaternion algebra are understood to be {\it right} modules unless explicit 
reference is made to left multiplication.

Consider a complex line bundle $\lambda$ over $M$, with $c_1(\lambda )=l$.
In order to use a spinor index in $\KO$-theory, we require that
$\rho_2 (l)=w_2(M)$.

Let $\Hh_\lambda$ be the bundle of quaternion algebras associated with
$\lambda$: $\Hh_\lambda = \Cc\oplus\lambda$, where $jz=\bar z j$ for
$j\in S(\lambda )$, $z\in\Cc$, and $j^2=-1$.
An $\Hh_\lambda$-vector bundle is a complex vector bundle by restriction of scalars.
We will say that a complex vector bundle $\zeta$
admits an $\Hl$-structure if it has a compatible $\Hl$-multiplication.
One can show that $\zeta$ has an $\Hl$-structure if and only if 
it has a non-singular skew-symmetric form $\zeta\otimes\zeta
\to \lambda$.
The next result,
which goes back to an idea of Dupont \cite{JLD}, is fundamental.
\begin{lemma}\label{7A}
Let $\beta$ be a $3$-dimensional complex vector bundle over $M-B$ such
that $c_1(\beta )=l$ and $c_3(\beta )=0$.
Then $\beta$ has a nowhere-zero section over $M-B$
and so splits as a direct sum $\beta = \Cc \oplus \mu$,
where $\mu$ is an $\Hh_\lambda$-line bundle.
\end{lemma}
\begin{proof}
We have to show that the sphere bundle
$S(\beta )$ has a section over $M-B$. The
top obstruction is a class $\o$ in $H^7(M-B;\, \Zz /2)$
($=H^7(M;\,\Zz /2)$), because $\pi_6(S^5)=\Zz /2$.
Consider an element $x$ of $H^1(M;\, \Zz /2)$, represented by a real line
bundle $\nu$. The zero-set of a generic cross-section of $\nu$ is
a $7$-dimensional submanifold $N$ of $M-B$.
It suffices to prove that $\o$ maps to zero in 
$H^7(N;\, \Zz /2)$. For evaluation on the fundamental class $[N]$
gives $(\o \cdot x)[M]\in\Zz /2$.
This is established by a $\KO$-theory Hopf theorem.
It will be convenient, as in \cite{MCC2} and
\cite{MCC3}, to use
a local coefficient notation $\KO^*(M;\, \alpha )$ for the
reduced $\KO$-theory $\rKO^*(M^{\alpha}) $ of the Thom space of
a (virtual) real vector bundle $\alpha$ over $M$.

The composition
$$
\KO^0(N;\, -\beta )  \xrightarrow{\eta (\nu )} \KO^0(N;\, \nu -\beta )
 \xrightarrow{\pi_!} \KO^{-2}(*)=\Zz /2,
$$
where $\eta (\nu )\in \KO^0(N;\, \nu )$ is a twisted Hopf element
(lifting the generator $\eta$ of $\KO^{-1}(*)=\Zz /2$)
and $\pi_!$ is the index determined by the spin structure on
$\tau  M-\beta$, maps the $\KO$-Euler class $\gamma (\beta )$ to
$(\o \cdot x)[M]\in\Zz /2$.
Now there is a $\Zz /2$-equivariant lifting:
$$
\KO^0_{\Zz /2}(N;\, -L\otimes\beta ) \xrightarrow{\eta (\nu )} 
\KO^0_{\Zz /2}(N;\, \nu -L\otimes\beta )
\xrightarrow{\pi_!} \KO^{-8}_{\Zz /2}(*;\, -6L)=\Zz , 
$$
where $L$ is the non-trivial $1$-dimensional real representation of
$\Zz /2$. 
(For more details, see \cite{MCC2}.)
But $\eta (\nu )$ is a torsion class. So the image of
$\gamma (L\otimes\beta )$ in $\Zz$ must be zero. 

Of course, a $2$-dimensional complex bundle $\mu$ with $c_1=l$ admits an
$\Hh_\lambda$-structure given by the exterior product to the determinant
bundle regarded as a skew-symmetric form 
$\mu\otimes\mu\to\Lambda^2\mu\cong \lambda$.
\end{proof}

\section{Reduction from $\U (4)$ to $Sp(2)\times_{\{\pm 1\}} U(1)$}
Let $\zeta$ be a $4$-dimensional complex vector
bundle over $M$. The structure group $U(4)$ of  
bundle $\zeta$ reduces to the subgroup 
$Sp(2)\times_{\{\pm 1\}} U(1)$ if and only if there is a complex line bundle
$\lambda$ over $M$ such that $\zeta$ has an  $\Hh_\lambda$-structure.
For more details see \cite{CCV}.

So consider a complex line bundle $\lambda$ with $c_1(\lambda )=l$ and 
$\rho_2(l)=w_2(M)$. We ask when the structure group of $\zeta$ can be reduced 
to an $\Hh_\lambda$-structure. Simultaneously, we want to characterize the 
Chern classes of such vector bundles.

If an $\Hl$-vector bundle has a nowhere zero section, we can split off a trivial
$\Hl$-line bundle. 
As a first consequence, an $m$-dimensional $\Hh_\lambda$-bundle over a $3$-dimensional space must
be trivial, that is, isomorphic to $\Hh_\lambda^m$, and so has $c_1=ml$,
where $l=c_1(\lambda )$.

Again we can treat the problem in three stages, starting with the stable
question over $M-B$. 
An $\Hh_\lambda$-bundle over $M-B$ can be reduced to dimension $1$ over
 $\Hh_\lambda$ because a higher dimensional bundle has a nowhere zero section.
So a necessary condition for lifting $\zeta$ is that 
$c_3(\zeta -\Hh_\lambda )=0$, that is, 
$$
c_3(\zeta )-lc_2(\zeta )+ l^3=0.
$$
In that case, we show that $(\zeta -\Hh_\lambda )\, |\, M-B$ 
can be reduced to complex dimension $2$ and so it has an 
$\Hh_\lambda$-structure.
Let $\beta$ be a $3$-dimensional complex vector bundle over $M-B$
such that $\beta\oplus\lambda$ is stably isomorphic
to $\zeta$. (Such a bundle always exists.) Note that $c_1(\beta )=l$.
We are assuming that $c_3(\beta )=0$, so that $\beta$ has
a nowhere-zero cross-section over $M-B$ by Lemma \ref{7A}.
This establishes:
\begin{lemma}\label{8A}
A $4$-dimensional complex bundle $\zeta$ has a stable $\Hh_\lambda$-structure 
over $M-B$ if and only if $c_1(\zeta)=2l$ and $c_3(\xi )-lc_2(\xi )+ l^3=0$.
\end{lemma}

The rest of the argument follows the familiar pattern.
The stable $\Hh_\lambda$-bundles are classified by a twisted $K$-group
$\KSp_\lambda^0(M)$, which can be
identified with the $\KO$-group $\rKO^{-2}(M^\lambda )$ of the
Thom space of $\lambda$.
(See \cite{MFA-EGR} pp. 135--136 and \cite{MCC1} Chapter 9
for the twisted $K$-theory.)
To describe the isomorphism, we think of 
$\rKO^{-2}(M^\lambda )$ as the $\KO$-theory with compact supports
$\KO^0(\Cc\oplus\lambda )$ of the total space of the bundle
$\Cc\oplus\lambda =\Hh_\lambda$. The isomorphism 
maps the class $[\mu]\in K\Sp_{\lambda}(M)$ of an $\Hl$-module $\mu$
over $M$ to the class in $\KO(\Cc\oplus\lambda)$
given by the $\Rr$-linear map
$$
\mu\to\mu:\quad x\mapsto xv
$$
over $v\in\Hh_\lambda=\Cc\oplus\lambda$, which is an isomorphism outside the 
(compact) zero-section.

A spin structure on $\tau M+\lambda$ gives an index, or Gysin map
$\KSp_\lambda^0(M)=\rKO^{-2}(M^\lambda )\to \KO^{-4}(*)=\Zz$.
The corresponding splitting of the complex $K$-groups is that
defined by taking the spin${}^c$ structure on $M$ with characteristic class
$c=-l$. 

Now we have a map from the $\KSp_\lambda$-sequence to the complex $K$-theory
giving a commutative diagram
$$
\begin{matrix}
\Zz = \KSp^0_\lambda (M,M-B)&\xrightarrow[\leftarrow - -]{}
&\KSp^0_\lambda (M) &\to &\KSp^0_\lambda (M-B)&\to &0\\
\,\,\downarrow\,\,2\times & & \downarrow && \downarrow && \\
\Zz = \K^0(M,M-B)&\xrightarrow[\leftarrow - -]{}
&\K^0(M) &\to &\K^0 (M-B)&\to &0
\end{matrix}
$$
in which the restriction from the twisted quaternionic to the complex theory
is multiplication by $2$ on the top cell $\Zz$.

Taking the spin${}^c$ structure on $M$ with characteristic class
$c=-l$, for $\zeta$ to admit an $\Hh_\lambda$-structure we, thus, require that
$$
\ind \zeta=(\e^{-l/2}\hat{\mathcal A}(\tau M)\ch (\zeta ))[M]\in 2\Zz ,
$$
since two stably isomorphic $\Hh_\lambda$-bundles of dimension $2$ over
$M$ are isomorphic.
This leads to the following characterization.
\begin{prop}\label{8B}
Let $\lambda$ be a complex line bundle over $M$ with $c_1(\lambda)=l$ and
$\rho_2(l)=w_2(M)$. Consider cohomology classes $u\in H^4(M\ ;\Zz)$
and $w\in H^8(M\ ;\Zz)$. Then there is a $2$-dimensional  
$\Hh_{\lambda}$-vector bundle $\zeta$ over $M-B$ with $c_2(\zeta)=u$ if and 
only if ${\rm Sq}^2\rho_2u=\rho_2(lu-l^3)$. In that case $c_1(\zeta)=2l$,
$c_3(\zeta)=lu-l^3$ and the expression
\begin{equation*}
K=\frac{1}{4}(p_1(\tau M)u+2u^2+l^2(3l^2-p_1(\tau M)-5u))[M]\in \Qq
\end{equation*} 
is an integer. Such a bundle extends to an $\Hh_{\lambda}$-bundle over $M$
with $c_4(\zeta)=w$ if and only if
\begin{equation*}
w[M]\equiv K \ \mod{12}.
\end{equation*}
\end{prop}

\begin{proof}
A complex vector bundle over $M$ of dimension $4$ lies in the stable range, and so 
the problem is $K$-theoretic. Over $M-B$ the solution is given by Corollary \ref{2C}
and Lemma \ref{8A}. Now elements $y\in \K^0(M)$ having $\Hh_{\lambda}$-structure
and extending a given element in $\K^0(M-B)$ are classified by even integers
$\ind{y}$. One can compute that
\begin{equation*}
\ind{\zeta}-\ind{\Hh_{\lambda}^2}=\frac{1}{6}(K-w[M]),
\end{equation*}
and this completes the proof.
\end{proof}

As immediate corrollaries we get
\begin{cor}\label{8C}
Let $\xi$ be a $4$-dimensional complex bundle over $M$,
and let $\lambda$ be a complex line bundle,
with $c_1(\lambda )=l$, such that $\rho_2(l)=w_2 (M)$.
Then $\xi$ admits an $\Hh_\lambda$-structure if and only if
\begin{equation*}
c_1(\xi )=2l;\qquad 
c_3(\xi )-lc_2(\xi )+ l^3=0
\end{equation*}
and the expression
\begin{equation*}
K=\frac{1}{4}(p_1(\tau M)c_2(\xi )+2c_2^2(\xi )+
l^2(3l^2-p_1(\tau M)-5c_2(\xi)))[M]\in \Qq,
\end{equation*}
which is always  an integer (if the conditions above hold), satisfies
\begin{equation*}
K\equiv c_4(\xi)[M]\ \mod4.
\end{equation*}
\end{cor}

\begin{cor}\label{8D}
Let $\lambda$ be a complex line bundle over $M$ with $c_1(\lambda )=l$
and $\rho_2(l)=w_2(M)$. Let $u\in H^4(M;\, \Zz )$. Then there is an
$\Hh_\lambda$-line bundle $\eta$ over $M$ with ($c_1(\eta )=l$ and)
$c_2(\eta )=u$ if and only if ${\rm Sq}^2\rho_2(u)=\rho_2(lu)$ and
the integer
$$
\frac{1}{4}(p_1(\tau M)u +2u^2 -l^2u)[M]\equiv 0\ \mod{12}.
$$
\end{cor}
\begin{proof}
If $\eta$ is such an $\Hh_\lambda$-line bundle, then
$\zeta = \eta \oplus\Hh_\lambda$ is a $4$-dimensional complex vector
bundle with $c_1(\xi )=2l$, $c_2(\xi )=u+l^2$, $c_3(\xi )=lu$,
$c_4(\xi )=0$.  Conversely, if a $4$-dimensional complex vector bundle
$\zeta$ has these Chern classes and admits an $\Hh_\lambda$-structure,
then, since its Euler class is zero, it splits as a direct sum of
an $\Hh_\lambda$-line and a trivial $\Hh_\lambda$ summand.
\end{proof}

\section{Reduction to $\Spinc (4)$ and $\Spinc (3)$}
Again let $\lambda$ be a complex line bundle over $M$ with
$c_1(\lambda )=l$ and $w_2(M)=\rho_2(l)$. 
We shall use the results of the last section on $\Hh_\lambda$-line
bundles to investigate the existence of $\Spinc (4)$ and
$\Spinc (3)$-structures on an $8$-dimensional vector bundle $\xi$.

Since $\Spinc (4)=(Sp(1)\times Sp(1)\times U(1))/\{\pm(1,1,1)\}$ any 
$4$-dimensional 
real vector bundle $\eta$ over $M$ having a $\Spinc (4)$-structure
with characteristic class $l$ can be expressed as
$$
\eta =\Hom_{\Hh_\lambda}(\zeta_-,\zeta_+)
=\zeta_+\otimes_{\Hh_\lambda} \zeta_-^*,
$$
where $\zeta_+$ and $\zeta_-$ are right $\Hh_\lambda$-line bundles over $M$
and $\zeta_-^*=\Hom_{\Hl}(\zeta_-,\Hl)$ is considered as a left 
$\Hl$-line bundle. 

We may also think of $\eta$ as the real form of the complex bundle 
$\zeta_+\otimes_{\Cc}\zeta_-^*$
with the real structure given by the symmetric $\Cc$-valued bilinear form
arizing as the tensor product of a $\lambda$-valued skew-symmetric form
on $\zeta_+$ and a $\lambda^*$-valued skew-symmetric form on $\zeta_-^*$.
\begin{lemma}\label{11A}
The characteristic classes of $\eta$ are:
$$
\begin{matrix}
w_2(\eta ) = \rho_2(l);\qquad & w_4(\eta )=w_4(\zeta_+)+w_4(\zeta_-) ;\\
q_1(\eta ;\, l)=-(c_2(\zeta_+)+c_2(\zeta_-)) ;\qquad &
e(\eta )=c_2(\zeta_+)-c_2(\zeta_-),
\end{matrix}
$$
with the appropriate choice of orientation.
\end{lemma}
\begin{proof}
Since $c_1(\zeta_+)=c_1(\zeta_-)=l$, by a splitting principle  we may suppose 
that
$\zeta_+ = \mu_+\oplus \lambda\otimes \mu_+^*$ and
$\zeta_- = \mu_-\oplus \lambda\otimes \mu_-^*$, where
$\mu_+$  and $\mu_-$ are complex line bundles.
Using the description of $\Cc\otimes_{\Rr}\eta$ as $\zeta_+\otimes_{\Cc}\zeta_-^*$,
we may identify $\eta$ with the complex bundle $(\mu_+^*
\oplus \lambda^*\otimes_{\Cc}\mu_+)\otimes_{\Cc}\mu_-$.  
\end{proof}
\begin{prop}\label{11B}
{\rm (\v Cadek and Van\v zura).}
Let $\xi$ satisfy $w_2(\xi )=w_2(M)=\rho_2(l)$. Then the structure group
of $\xi$ can be reduced to $\Spinc (4)$ with characteristic class $l$
if and only if there exist classes $u_+,\, u_-\in H^4(M;\, \Zz )$
such that the following conditions are satisfied.
\begin{gather*} q_1(\xi ;\, l) = -(u_++u_-); \qquad w_6(\xi)=0;\qquad 
{\rm Sq}^2\rho_2u_+=\rho_2(lu_+);\\
e(\xi)=0;\qquad p_2(\xi )= (u_+-u_-)^2;\\
\frac{1}{4}(u_+(p_1(\tau M)+2u_+-l^2))[M]\equiv 0 \, \, \mod{12};\\
 \frac{1}{4}(u_-(p_1(\tau M)+2u_--l^2))[M]\equiv 0 \, \, \mod{4}.
\end{gather*}
\end{prop}

\begin{remark}\label{11C}
The first three conditions imply that
${\rm Sq}^2\rho_2(u_-)=\rho_2(lu_-)$. If the 
conditions
on $u_+$ and $u_-$ are satisfied, the last congruence holds also $\mod{12}$.
The integrality of $\ind (\xi-\Rr^8)\otimes\Cc-\ind (\lambda+\lambda^*-\Cc^2)$
gives
$$
(q_1(\xi;\, l)^2+q_1(\xi;\, l)l^2+p_2(\xi)-q_1(\xi;\, l)p_1(\tau M))[M]
\equiv 0 \, \, \mod{3},
$$
which is the sum of the $\mod{3}$ congruences for $u_+$ and $u_-$.

\end{remark}
\begin{proof}
Given $\Hh_{\lambda}$-line bundles $\zeta_+$ and $\zeta_-$ with $c_2$ equal to 
$u_+$ and $u_-$
respectively, we construct $\eta$ as above. Then $\eta\oplus\Rr^4$
is isomorphic to $\xi$ if and only if $q_1(\eta ;\, l)=q_1(\xi ;\, l)$,
$p_2(\xi )=p_2(\eta )=e(\eta )^2$ and $e(\xi )=e(\eta\oplus\Rr^4)=0$.
The existence of $\zeta_+$ and $\zeta_-$ is guaranteed by Corollary \ref{8D}.
\end{proof}

The group $\Spinc(3)$ is isomorphic to $\Sp (1)\times_{\{\pm 1\}} U(1)$. 
By a similar argument, for
any $3$-dimensional spin$^c$ vector bundle
$\alpha$ over $M$ with $w_2(\alpha )=\rho_2(l)$ there is just one 
$\Hh_{\lambda}$-line vector bundle $\zeta$
such that
$$
\Rr\oplus \alpha =\End_{\Hl}(\zeta)=\zeta\otimes_{\Hh_{\lambda}}\zeta^*.
$$ 
Since $w_4(\alpha)=0$, $q_1(\alpha;l)=2u$ for an 
element $u\in H^4(M;\, \Zz)$. From Lemma \ref{11A} 
and Corollary \ref{8D} 
we get

\begin{prop}\label{12A}
Let $w_2(M)=\rho_2(l)$ and let $u\in H^4(M;\, \Zz )$. Then there is a
$3$-dimensional vector bundle  $\alpha$ over $M$ with $w_2(\alpha)=\rho_2(l)$
and $q_1(\alpha,l )=2u$ if and only if ${\rm Sq}^2\rho_2(u)=\rho_2(lu)$ and
the integer
$$
\frac{1}{4}u(p_1(\tau M) -2u -l^2)[M]\equiv 0\ \mod{12}.
$$
\end{prop}

\begin{proof}
The conditions are necessary and sufficient for the existence of an 
$\Hh_{\lambda}$-line bundle $\zeta$ with $c_1(\zeta)=l$ and $c_2(\zeta)=-u$. 
Then $\alpha\oplus\Rr=\zeta\otimes_{\Hh_{\lambda}}\zeta^*$ has the prescribed 
characteristic classes by Lemma \ref{11A}.
\end{proof}

\begin{cor}\label{12B}
Let $\xi$ be an $8$-dimensional real vector bundle over $M$ with 
$w_2(\xi )=w_2(M)=\rho_2(l)$. Then the structure group
of $\xi$ can be reduced to $\Spinc (3)$ with characteristic class $l$
if and only if there exists a class $u\in H^4(M;\, \Zz )$
such that the following conditions are satisfied.
\begin{gather*} q_1(\xi ;\, l) = 2u; \qquad w_6(\xi)=0;\qquad 
{\rm Sq}^2\rho_2u=\rho_2(lu);\\
e(\xi)=0;\qquad p_2(\xi )= 0;\\
\frac{1}{4}(u(p_1(\tau M)-2u-l^2))[M]\equiv 0 \, \, \mod{4}.
 \end{gather*}
\end{cor}

\section{Reduction to $\Sp (2)\times_{\{\pm 1\}} U(1)$ and 
$\Sp (2)\times_{\{\pm 1\}}\Sp (1)$ : a quaternionic structure}
It is straightforward now to give conditions for an $8$-dimensional real
vector bundle $\xi$ over $M$ to admit an $\Hh_\lambda$-structure.
\begin{prop}\label{9A}
{\rm (\v Cadek and Van\v zura).}
Let $\lambda$ be a complex line bundle over $M$ with
$c_1(\lambda )=l$ and $w_2(M)=\rho_2(l)$.
An oriented $8$-dimensional real vector bundle $\xi$ admits 
an $\Hh_\lambda$-structure
if and only if:
\begin{align*}
\text{\rm a)}\ & w_2(\xi )=0\quad\text{and}\quad w_6(\xi )=w_2(M)(w_4(\xi)
+w_2^2(M));\\
\text{\rm b)}\ & \frac{1}{2}(p_2(\xi)-q_1^2(\xi))[M]\equiv \frac{1}{4}
(2q_1^2(\xi)-p_1(\tau M)q_1(\xi)+l^2(p_1(\tau M)-3q_1(\xi)+l^2))[M]\ \mod{4};
\\
\text{\rm c)}\ & 2e(\xi)=p_2(\xi)-q_1^2(\xi). 
\end{align*}
\end{prop}
\begin{proof}
Suppose that there is a 4-dimensional complex $\Hh_{\lambda}$-vector bundle
$\zeta$ with $c_2(\zeta)=u$ and $c_4(\zeta)=w$ stably isomorphic to $\xi$. Then
\begin{align*}
w_2(\xi)&=\rho_2(c_1(\zeta))=\rho_2(2l)=0,\\
u&=-q_1(\xi;2l)=2l^2-q_1(\xi),\\
w_6(\xi)&=\rho_2(c_3(\zeta))=\rho_2(lu-l^3)=w_2(M)(w_4(\xi)+w_2^2(M)),\\
2w&=p_2(\xi)-c_2^2(\zeta)+2c_1(\zeta)c_3(\zeta)=p_2(\xi)-q_1^2(\xi).
\end{align*}

According to Proposition \ref{8B} necessary and sufficient conditions for 
the existence of a stable 
$\Hh_{\lambda}$-structure on $\xi$ are then given by conditions a) and b)
except that the congruence in (\ref{8B}) is required
$\mod{12}$. But the congruence is always satisfied $\mod{3}$ since the 
integrality of $\ind(\xi\otimes \Cc-\Cc^8)-
\ind(2(\lambda\oplus\bar{\lambda})-\Cc^4)$ gives
$$(4q_1^2(\xi)-2p_2(\xi)-q_1(\xi)p_1(\tau M)+l^2p_1(\tau M)-2l^4)[M]\equiv
0\, \, \mod{3}.$$ 
 
Condition c) expresses the equality of the Euler classes.
\end{proof}

We next consider the extension to other twisted quaternionic structures.
Let $\alpha$ be a $3$-dimensional oriented Euclidean
vector bundle over $M$. 
Recall that an oriented $3$-dimensional real inner product space $V$ has
a vector product $\times : V\otimes V\to V$, which can be used to give
$\Rr 1\oplus V$ a quaternionic multiplication with
$(0,x)\cdot (0,y) = (-\langle x,y\rangle, x\times y)$ for $x,\, y\in V$.
Applied to the fibres of $\alpha$, this gives
the $4$-dimensional real vector bundle $\Rr 1\oplus\alpha$ the structure
of a bundle of quaternion algebras.

When $\alpha =\Rr\i\oplus\lambda$, we may identify $\Rr 1\oplus\alpha$ with
the bundle $\Hh_\lambda$ already considered.
We now investigate the more general problem of the existence of
a right $\Rr 1\oplus\alpha$-module 
structure on an $8$-dimensional real vector bundle $\xi$.
This is equivalent to reduction of the structure group to 
$\Sp (2) \times_{\{\pm 1\}}\Sp (1)$. 
To apply the methods of this paper we must require that
$w_2(\alpha)=w_2(M)=\rho_2(l)$.
In that case, the bundle of algebras $\Rr 1\oplus\alpha$ is the 
endomorphism algebra of an $\Hh_\lambda$-line bundle $\zeta$:
$$
\Rr 1\oplus \alpha =  \End_{\Hh_\lambda} (\zeta ).
$$
Notice that 
$$
\End_{\Hh_\lambda}(\Hh_{\lambda}) = \Hh_{\lambda}.
$$

\begin{lemma}\label{13A} Let $\xi$ be an $8$-dimensional real 
vector bundle and let $\zeta$ be an $\Hh_\lambda$-line bundle. 
Then $\xi$ has an $\End_{\Hh_\lambda}(\zeta )$-structure 
if and only if there is a $2$-dimensional  
$\Hh_{\lambda}$-vector bundle $\eta$ such that
$$
\xi\cong\Hom_{\Hh_{\lambda}}(\zeta ,\eta)=
\eta\otimes_ {\Hh_{\lambda}}\zeta^*.
$$
\end{lemma}

\begin{proof} This is standard module theory: $\zeta$ is an invertible 
$(\End_{\Hl}(\zeta),\Hl)$-bimodule.
\end{proof}
\begin{lemma}\label{13B}
The characteristic classes of $\xi\ = \eta\otimes_ {\Hh_{\lambda}}\zeta^*$ 
where $c_1(\zeta^*)=-l$, $c_2(\zeta)=-u$ and $c_1(\eta)=2l$ are
\begin{gather*}
w_2(\xi )=0; \qquad 
q_1(\xi)= 2u-c_2(\eta)+2l^2;\\
p_2(\xi) = 2c_4(\eta)+c_2(\eta)(c_2(\eta)-2u-4l^2)+6u^2+6ul^2+4l^2;\\
e(\xi)=c_4(\eta)+c_2(\eta)u-l^2u+u^2.
\end{gather*}
\end{lemma}

\begin{proof}
To compute the characteristic classes we can suppose that $\eta$
is a sum of two $\Hl$-line bundles $\eta_1\oplus\eta_2$ and use Lemma \ref{11A}.
For instance, the computation of the Euler class is as follows:
\begin{align*}
e(\xi)&=e((\eta_1\oplus\eta_2)\otimes_{\Hl}\zeta^*)=
e(\eta_1\otimes_{\Hl}\zeta^*)e(\eta_2\otimes_{\Hl}\zeta^*)\\
&=(c_2(\eta_1)-c_2(\zeta^*))(c_2(\eta_2)-c_2(\zeta^*))
=c_4(\eta)+c_2(\eta)u-l^2u+u^2
\end{align*}
\end{proof}

\begin{prop}\label{13C}
{\rm (See [5], Theorem 8.1).}
Let $w_2(M)=\rho_2(l)$. 
Then an $8$-dimensional oriented vector bundle $\xi$ has
an $\Hh_{\alpha}$-structure with $w_2(\alpha)=\rho_2(l)$ and
$q_1(\alpha;l)=2u$ if and only if $w_2(\xi)=0$ and  
the following conditions are satisfied.
\begin{equation*} p_2(\xi)-q_1^2(\xi)-2e(\xi)=0;\qquad  
w_6(\xi)+w_4(\xi)\rho_2(l)+\rho_2(l^3)=0;
\qquad {\rm Sq}^2\rho_2u=\rho_2(lu);
\end{equation*}
\begin{equation*}
\frac{1}{4}(u(p_1(\tau M)-2u-l^2))[M]\equiv 0 \, \, \mod{12};\\
\end{equation*}
\vskip-1.5\baselineskip
\begin{multline*}
\frac{1}{4}(4q_1^2(\xi)-2p_2(\xi)-q_1(\xi)p_1(\tau M)-3l^2q_1(\xi)+
l^2p_1(\tau M)+l^4\\
+u(20u+10l^2+2p_1(\tau M))-12q_1(\xi)u)[M]\equiv 0 \, \, \mod{4}.
\end{multline*} 
\end{prop}

Substituting $u=0$ we get the conditions from Proposition \ref{9A}.

\begin{example}
The quaternionic projective space $\mathbb HP^2$, the Grassmannian 
$G_{4,2}(\Cc)$ and the homogeneous space $G_2/SO(4)$ are known to be
quaternionic Kaehler manifolds and therefore the structure groups of 
their tangent bundles can be reduced to $Sp(2)\times_{\{\pm 1\}}Sp(1)$.
Using our results we can say a little bit more.

The tangent bundle $\tau(\Hh P^2)$ has no complex structure (and hence no
$\Hl$-structure), but it has an $\mathbb H_{\alpha}$-structure if and only if
$q_1(\alpha,0)=2ka$, where $k\equiv 1,9$ (mod 24) and $a$ is the generator  of $H^4(\Hh P^2;\Zz)$.

 There are infinitely many complex structures on $\tau (G_{4,2}(\mathbb C))$, but
none of them extends to an $\mathbb H_{\lambda}$-structure.
The tangent bundle admits  $\mathbb H_{\alpha}$-structures with
$q_1(\alpha,0)=2ka^2$, where $k\equiv 1,9$ (mod 12) and $a$ generates $H^2(G_{4,2};\Zz)$.

In the case of $\tau(G_2/SO(4))$ our Proposition \ref{13C} can be used 
only to decide if there is an $\mathbb H_{\alpha}$-structure with 
$w_2(\alpha)=w_2(G_2/SO(4))=0$. And such a structure does not exist.

Consider next the complex manifolds  $V_d=\{(z_0,\cdots,z_5)\in\mathbb CP^5;\
z_0^d+z_1^d+\cdots+z_5^d=0\}$. Among them only $V_2=G_{4,2}(\Cc)$ and $V_6$ have
almost quaternionic structures (see also \cite{CV4}). 
The tangent bundle $\tau(V_6)$, like $\tau(G_{4,2})$,
has infinitely many complex structures, none  being a restriction of an $\mathbb H_{\lambda}$-structure.
But $\tau(V_6)$ is an $\mathbb H_{\alpha}$-module for the bundles $\alpha$
with $q_1(\alpha,0)=2ka^2$, where $k\equiv 1$ (mod 4) and $a$ is a generator 
of $H^2(V_6;\Zz)$.
\end{example}

\section{Reduction to $\Spinc (5)$ : a $3$-field}
Consider the composition
$$
\Sp(2)\times_{\{\pm 1\}} U(1) =\Spinc (5) \to \SO (5) \to \SO (8)
$$
of the canonical map and the standard inclusion.
Suppose that $\lambda$ is a complex line bundle over $M$ and that
$\zeta$ is a $2$-dimensional $\Hh_\lambda$-bundle. Then the
$6$-dimensional complex bundle $\lambda^*\otimes\Lambda^2 \zeta$ 
has a trivial $1$-dimensional summand given by the (dual of the)
skew form $\zeta\otimes\zeta\to\lambda$ and a real structure given by 
a symmetric $\Cc$-valued form on $\lambda^*\otimes\Lambda^2 \zeta$.
So we may write $\lambda^*\otimes _{\Cc}\Lambda^2 \zeta =\Cc\otimes_{\Rr} (\Rr\oplus\eta )$
for a $5$-dimensional real vector bundle $\eta$. 

\begin{lemma} \label{X1}
The characteristic classes of $\eta$ are:
$$
\begin{matrix}
w_2(\eta )=\rho_2(l); \qquad & w_4(\eta )=\rho_2(l^2)+w_4(\zeta ) ;\\
q_1(\eta ;\, l)= l^2-c_2(\zeta ): \qquad &
p_2(\eta ) = q_1^2(\eta;\, l) -4c_4(\zeta ) .
\end{matrix}
$$
\end{lemma}
\begin{proof}
Apply Lemma \ref{5A} to $\eta\oplus\Rr$. 
\end{proof}

\begin{prop}\label{X2}
{\rm (Crabb and Steer).}
Suppose that $w_2(M) =w_2(\xi )=\rho_2(l)$.
Then $\xi$ admits a $\Spinc (5)$-structure with characteristic class
$l$ if and only if the following conditions hold.
\begin{gather*}
w_6(\xi )=0;\qquad e(\xi )=0;\\
(\frac{1}{2}(p_2(\xi)-q_1^2(\xi;l))+
\frac{1}{2}(2q_1^2(\xi;\, l)-q_1(\xi;\, l)p_1(\tau M)+q_1(\xi;\, l)l^2))[M]
\equiv 0 \, \, \mod{8}.
\end{gather*}
\end{prop}
\begin{proof} The proof follows the same pattern as that of 
Proposition \ref{5B}, only that instead of Proposition \ref{3E} we use Proposition
\ref{8B}. (If $e(\xi)=0$, the last congruence is always satisfied $\mod{3}$
by (\ref{4B}).) 
\end{proof}

Next we describe when a given $8$-dimensional vector bundle has a 
$3$-dimensional subbundle. To do so, we will need the following statement:

\begin{prop}\label{12C}
Let $w_2(M)=\rho_2(l)$. Consider cohomology classes $u\in H^4(M;\, \Zz )$
and $z\in H^8(M;\, \Zz)$. Then there is a $5$-dimensional vector bundle $\beta$
over $M$ with $w_2(\beta)=\rho_2(l)$, $q_1(\beta;l)=u$ and $p_2(\beta)=z$
if and only if
${\rm Sq}^2\rho_2(u)=\rho_2(lu)$ and there is a class $w\in H^8(M;\ \Zz)$
such that $4w=u^2-z$ and the integer
$$
w[M]\equiv \frac{1}{4}u(2u-p_1(\tau M) +l^2)[M]\quad \mod{12}.
$$
\end{prop}

\begin{proof}
This follows from Lemma \ref{X1}
and Proposition \ref{8B}.
\end{proof}

\begin{cor}\label{12D}
{\rm (\v Cadek, Van\v zura).}
Let $w_2(M)=\rho_2(l)$ and let $\xi$ satisfy $w_2(\xi )=0$. Then  $\xi$ has a 
$3$-dimensional subbundle $\alpha$ with characteristic classes 
$w_2(\alpha) =\rho_2(l)$ and $q_1(\alpha;l)=2u$
if and only if  the following conditions are satisfied.
\begin{gather*} e(\xi)=0; \qquad w_6(\xi)+w_4(\xi)\rho_2(l)+\rho_2(l^3)=0;
\qquad 
{\rm Sq}^2\rho_2u=\rho_2(lu);\\
\frac{1}{4}(u(p_1(\tau M)-2u-l^2))[M]\equiv 0 \, \, \mod{12};\\
\end{gather*}
\vskip-3\baselineskip
\begin{multline*}
\frac{1}{4}(p_2(\xi)+q_1^2(\xi)-q_1(\xi)p_1(\tau M)-3l^2q_1(\xi)+l^2p_1(\tau M)
+l^4\\
+u(20u+10l^2+2p_1(\tau M)-12q_1(\xi))[M]\equiv 0 \, \, \mod{4}.
\end{multline*} 
\end{cor}

\begin{proof} The conditions on $u$ ensure the existence of a  3-dimensional 
vector bundle $\alpha$ with $q_1(\alpha;l)=2u$. It remains
to show that the other conditions ensure the existence of a 5-dimensional 
vector bundle $\beta$ with $w_2(\beta)=\rho_2(l)$, $q_1(\beta;-l)= q_1(\xi)-
q_1(\alpha;l)-l^2=q_1(\xi)-2u-l^2$ and $p_2(\beta)=p_2(\xi)-
p_1(\beta)p_1(\alpha)=p_2(\xi)-(4u+l^2)(2q_1(\xi)-4u-l^2)$. This can be done 
using the previous proposition. 
\end{proof}
\begin{remark}
In [4] a certain automorphism $\varphi$ of the group $\Spin (8)$ was found such 
that an 8-dimensional vector bundle $\xi$ with the spin structure 
$\tilde\varphi$ admits reduction to $\Sp (2)\times_{\{\pm 1\}}\Sp (1)$ 
if and only 
if a vector bundle $\eta$ with spin structure $\bar\eta=\varphi^*(\bar\xi)$
has a 3-dimensional subbundle ([4], Theorem 3.2). Moreover, $\varphi^*$ in 
the integer cohomology of $B\Spin (8)$ was computed as
$$
\varphi^*(q_1)=q_1;\qquad \varphi^*(e)=-q_2;\qquad \varphi^*(q_2)=-e.
$$
Here $q_2$ is uniquely defined by the equation $p_2=q_1^2+2e+4q_2$. This relates
reductions to $\Sp (2)\times_{\{\pm 1\}}\Sp (1)$ with reductions
to $\Spin (5)\times_{\{\pm 1\}}\Spin (3)$ and allows us to deduce 
Proposition \ref{13C} from Corrolary \ref{12D}.
\end{remark}

\end{document}